\documentclass[12pt]{article}
\usepackage{geometry}                
\usepackage{graphicx,color,mathtools}
\usepackage{amssymb,amsmath,amsthm,mathrsfs}
\usepackage[all,cmtip]{xy}
\usepackage{epstopdf, comment, url}
\usepackage{bm} 
\usepackage{enumitem}

\usepackage[pdftex,bookmarks,pdfnewwindow,plainpages=false,unicode,pdfencoding=auto]{hyperref}

\numberwithin{equation}{section}

\newtheorem{theorem}{Theorem}[section]

\newtheorem{lemma}[theorem]{Lemma}

\theoremstyle{remark}
\newtheorem*{remark}{Remark}

\theoremstyle{definition}

\DeclareMathOperator{\aut}{Aut}

\DeclareMathOperator{\inv}{Inv}
\DeclareMathOperator{\BMO}{BMO}
\DeclareMathOperator{\GCE}{GCE}
\DeclareMathOperator{\inn}{Inn}

\title{Critical structures of inner functions II}
\author{Oleg Ivrii}
\date{September 8, 2026}

\begin{document}

\maketitle

\begin{abstract}
We study the correspondence proposed in \cite{critical-structures} between inner functions modulo post-compositions with automorphisms of the unit disk and cyclic subspaces of the weighted Bergman space $A^2_1$. The correspondence sends an inner function $I$ to the invariant subspace $[I']$, and in the opposite direction, assigns to a non-zero function $H \in A^2_1$ the Liouville map $I_H$ associated to the canonical solution of the Gauss curvature equation $\Delta u = |H|^2 e^{2u}$. We prove that $I'_H$ generates the same cyclic subspace as $H$. Combined with the earlier results in \cite{critical-structures}, this shows that these two mappings are inverses of one another, and hence the correspondence $I \to [I']$ is a bijection. 
\end{abstract}

\section{Introduction}

An {\em inner function} $F$ is a holomorphic self-map of the unit disk such that for a.e.~$\theta \in [0,2\pi)$, the radial boundary value $F(e^{i\theta}) = \lim_{r \to 1} F(re^{i\theta})$ exists and belongs to the unit circle.
By a classical result of A.~Beurling, one can parametrize inner functions up to post-compositions with rotations by invariant subspaces of $H^2$.
In this paper, we continue the programme started in \cite{critical-structures} of parametrizing inner functions up to post-compositions with M\"obius transformations by cyclic subspaces of $A^2_1$.

The weighted Bergman space $A^2_1$ consists of holomorphic functions on the unit disk for which
$$
\| H \|_{A^2_1}^2 \, = \, \frac{2}{\pi} \int_{\mathbb{D}} |H(z)|^2 (1-|z|^2) dA(z) \, < \, \infty.
$$
We say that a (closed) subspace $X \subset A^2_1$ is {\em invariant} if $zX \subset X$.
For a function $H \in A^2_1$, the {\em cyclic subspace generated by $H$} is defined as the minimal invariant subspace of $A^2_1$
which contains $H$:
$$
[H] = \overline{\{ Hp \, : \, p \text{ polynomial} \}}^{A^2_1}.
$$
In \cite{critical-structures}, the author set out to show that inner functions
$$
\inn / \aut(\mathbb{D}) = \{ \text{cyclic subspace of }A^2_1 \}, \qquad I \to [I'],
$$
but only proved that the map is well-defined and injective. Here, we prove surjectivity:

\begin{theorem}
\label{main-thm}
For every $H \in A^2_1 \setminus \{0\}$, there is an inner function $I_H$ whose derivative generates $[H]$. Moreover, $I_H$ is unique up to post-composition with an automorphism of the unit disk.
\end{theorem}

\subsection{Canonical solutions}
\label{sec:canonical-solutions}

The inner functions $I_H$ are constructed using the Liouville correspondence which provides a bridge between complex analysis and semi-linear elliptic PDE.

For a function $H \in A^2_1 \setminus \{ 0 \}$, we consider the Gauss curvature equation
\begin{equation}
\tag{$\GCE_H$}
\Delta u = |H|^2 e^{2u}.
\end{equation}
By Liouville's theorem \cite[Theorem 3.3]{KR08}, every solution of $\GCE_H$ has the form
$$
u = \log \frac{1}{|H|} \frac{2 |F'|}{1-|F|^2},
$$
where $F$ is a holomorphic self-map of the unit disk which has the same critical points as $H$, counted with multiplicity. Furthermore, $F$ is unique up to post-composition with an automorphism of the unit disk. We refer to $F$ as the {\em Liouville map} of $u$.

By the theory of semi-linear elliptic PDEs, for every $h \in L^\infty(\partial \mathbb{D})$, the boundary value problem
\begin{equation}
\label{eq:BVP}
\begin{cases}
\Delta u = |H|^2 e^{2u}, \qquad & \text{in }\mathbb{D}, \\
u = h, \qquad & \text{on }\partial\mathbb{D},
\end{cases}
\end{equation}
has a unique solution. Here, the boundary values are understood in the sense of boundary traces: as $r \to 1$, the measures $u(re^{i\theta}) d\theta \to h d\theta$ converge in the weak-$*$ topology. Furthermore, if $h_1 \le h_2$ pointwise a.e.~then, $u_{H, h_1} \le u_{H, h_2}$.

For $n \in \mathbb{Z}$, we write $u_{H, n}$ for the solution of ($\GCE_H$) with constant boundary values $n$.
By the monotonicity of solutions, $u_{H, m} \le u_{H, n}$ if $m \le n$. It is explained in \cite{critical-structures} that
the pointwise limit $u_{H, \infty} = \lim u_{H, n}$ is also solution, called the {\em canonical solution}\/.

We write $I_H$ for the Liouville map associated to $u_{H,\infty}$. In \cite{critical-structures}, the following observations were made regarding $I_H$:

\begin{itemize}
\item $I_H$ is an inner function.
\item $I'_H \in [H]$.
\item If $[H_1] = [H_2]$ then $I_{H_1} = I_{H_2}$.
\item $[I'] = [(m \circ I)']$ for any $m \in \aut(\mathbb{D})$.
\item If $F$ is an inner function, then $I_{F'} = F$.
\end{itemize}

What is missing is that $I'_H$ generates $[H]$. This is the content of Theorem \ref{main-thm}.

\section{Preliminaries}

In this section, we gather some miscellaneous lemmas that will be used throughout this paper. We start with the following standard lemma:

\begin{lemma}
\label{bounded-lemma}
Let $X$ be an invariant subspace of $A^2_1$. If $f \in X$ and $b \in H^\infty$, then $fb \in X$.
\end{lemma}

The lemma can be proved by first approximating $b(z)$ in $A^2_1$ by its dilates $b(rz)$ with $0 < r < 1$, and then approximating each dilate uniformly on the closed unit disk by polynomials.

\begin{lemma}
\label{beta-lemma}
For every holomorphic self-map $J$ of the unit disk and $0 \le \beta < 2$, we have
$$
\int_{\mathbb{D}} \frac{|J'(z)|^2}{(1-|J(z)|^2)^\beta} \, (1-|z|^2) dA(z) < \infty.
$$
\end{lemma}

\begin{proof}
Since increasing $\beta$ increases the integrand, we may assume that
$1<\beta<2$. Consider the function
$$
v(z) = 1 - (1 - |J(z)|^2)^{2-\beta}.
$$
Differentiating, we get
\begin{equation}
\label{eq:deltav}
\Delta v \ge 4(2-\beta)(\beta-1) \cdot |J'|^2 (1-|J|^2)^{-\beta}.
\end{equation}
Thus, $v$ is a subharmonic function on the unit disk which takes values between 0 and 1. Applying the Poisson-Jensen formula for subharmonic functions \cite[Theorem 4.5.1]{ransford} on $B(0,r)$ and taking $r \to 1$ yields
$$
\frac{1}{2\pi} \int_{\mathbb{D}} \Delta v(z)  \log \frac{1}{|z|} \, dA(z) \le 1.
$$
The lemma follows after substituting the lower bound for $\Delta v$ from (\ref{eq:deltav}) into the preceding formula and using $1-|z|^2 \le 2 \log(1/|z|)$.
\end{proof}

\begin{remark}
From the Poisson-Jensen formula, it follows that when $\beta = 2$, the above
integral is finite if and only if 
$$
 \int_{\partial \mathbb D} \log \frac{1}{1-|J(\zeta)|^2}\,d\zeta < \infty.
$$
 By a result of de Leeuw and Rudin, this occurs
precisely when $J$ is a non-extreme point of the unit ball of $H^\infty$,
see \cite[Section 5]{dLR58} or \cite[Chapter 9]{Hof62} for details.
\end{remark}

The following lemma is a simple consequence of the Schwarz lemma:

\begin{lemma}
\label{flexibility}
For every $R>0$, there exists a constant $C=C(R)>0$ such that
for any holomorphic self-map $J$ of the unit disk,
$$
1/C \, \le \, \frac{1 - |J(z_1)|^2}{1-|J(z_2)|^2} \, \le \, C, \qquad z_1, z_2 \in \mathbb{D}, \qquad d_{\mathbb D}(z_1,z_2)<R.
$$
\end{lemma}

Let $I \subset \partial \mathbb{D}$ be an arc on the unit circle of length $<1/10$. The Carleson square with base $I$ is defined as
$$
Q_I = \bigl \{ z \in \mathbb{D} : |z| > 1 - |I|, \ z/|z| \in I \bigr \}.
$$
Let $\xi_I$ be the midpoint of the arc $I$. We denote the center of $Q_I$ by $z_{Q_I} = (1-|I|/2)\xi_I$ and
 the side length of $Q_I$ by $\ell(Q_I) = |I|$.

\begin{lemma}
\label{nde}
For every holomorphic self-map $J$ of the unit disk and Carleson square $Q \subset \mathbb{D}$,
\begin{equation}
\label{eq:nde}
\frac{1}{\ell(Q)} \int_{Q} |J'(z)|^2 (1-|z|^2) \, dA(z) \le C(1 - |J(z_Q)|^2),
\end{equation}
where $C > 0$ is an absolute constant.
\end{lemma}

\begin{remark}
From the inclusion $H^\infty \subset \BMO$, it follows that the left hand side is bounded by an absolute constant. The lemma says that the
normalized $J$-energy of $Q$ is small if $|J(z_Q)|$ is close to 1.
\end{remark}

\begin{proof}
By the Poisson-Jensen formula, we have
\begin{align*}
|J(w)|^2 & = \frac{1}{2\pi} \int_{\partial \mathbb{D}} |J (\zeta)|^2 P(w, \zeta) |d\zeta| - \frac{1}{2\pi} \int_{\mathbb{D}} 
 \Delta |J(z)|^2 G(w,z) dA(z) \\
& \le 1 - \frac{2}{\pi} \int_{\mathbb{D}} 
|J'(z)|^2 G(w,z) dA(z).
\end{align*}
Rearranging, we get
$$
\int_{\mathbb{D}} |J'(z)|^2 G(w,z) dA(z) \lesssim 1-|J(w)|^2.
$$
Using the interpretation of the Green's function as the occupation density of Brownian motion, it is not difficult to see that
$$
G(w, z) \asymp \frac{1-|z|^2}{\ell(Q)}, \qquad w = (1 - 2\ell(Q)) \xi_I, \qquad z \in Q.
$$
Consequently,
$$
\frac{1}{\ell(Q)} \int_{Q} |J'(z)|^2 (1-|z|^2) \, dA(z) \lesssim 1-|J(w)|^2.
$$
Since the hyperbolic distance $d_{\mathbb{D}}(w, z_Q) = O(1)$, by Lemma \ref{flexibility}, the right hand side is comparable to $1 - |J(z_Q)|^2$.
\end{proof}

\section{A homotopy argument}

We will deduce Theorem \ref{main-thm} from the following result:

\begin{theorem}
\label{main-thm2}
Let $J, F$ be non-constant holomorphic self-maps of the unit disk which have the same critical points, counting multiplicity. Suppose that their quotient $q = F'/J'$ satisfies
\begin{equation}
\label{eq:qbound}
|q(z)| \le \frac{C}{1-|J(z)|^2}.
\end{equation}
Then, $F' \in [J']$.
\end{theorem}

\begin{remark}
(i) Let $\inv(H^\infty)$ be the collection of bounded analytic functions on the unit disk with bounded inverses.
The proof below will show that there exists a sequence of functions $a_k \in \inv(H^\infty)$ such that
$J' a_k \to F'$ in $A^2_1.$

(ii) Since the unit disk is simply-connected and $q$ is zero-free, it admits a holomorphic logarithm. We fix one such branch and define the fractional powers of $q$ by $q^t = e^{t\log q}$,\, $0 \le t \le 1$.
\end{remark}

We consider the path
$$
h_t = J' q^t, \qquad 0 \le t \le 1,
$$
which connects $h_0 = J'$ with $h_1 = F'$. Since $F' = J'q$, we have
$$
|h_t|^2 =  |J'|^{2(1-t)} |F'|^{2t}.
$$
By the weighted AM-GM inequality,
$$
|h_t|^2 \, \le \, (1-t) |J'|^2 + t |F'|^2 \, \le \, |J'|^2 + |F'|^2.
$$
Thus, $h_t \in A^2_1$ for every $0 \le t \le 1$ and
$$
\| h_t \|^2_{A^2_1} \le \| J' \|^2_{A^2_1} +  \| F' \|^2_{A^2_1} .
$$
As $t \to 1$, the functions $h_t$ converge pointwise to $h_1$. By Lebesgue's dominated convergence theorem, $h_t \to h_1$ in $A^2_1$
as $t \to 1$.
Since $[J']$ is closed, if we can prove
$$
h_t \in [J'], \qquad \text{for every }0 < t < 1,
$$
it would follow that $h_1$ is also in $[J']$.

Suppose we know that $h_t \in [J']$ and we want to show that $h_{t+s} \in [J']$ for some $s > 0$.
Unfortunately, we cannot simply multiply $h_t$ by $q^s$ since $q$ may be unbounded. Instead, we define the auxiliary functions
$$
a_{r,s}(z) = q(rz)^s, \qquad 0 < r < 1.
$$
Since $q$ is non-vanishing, each function $a_{r,s}$ is bounded with bounded inverse. We will show that
\begin{equation}
\label{eq:required-convergence}
h_t a_{r,s} \to h_{t+s}, \qquad \text{in }A^2_1,
\end{equation}
as $r \to 1$, provided that
\begin{equation}
\label{eq:st}
2s < 1-t.
\end{equation}
 In view of Lemma \ref{bounded-lemma}, this allows us to advance along the path $h_t$ in sufficiently small steps. Given any $t < 1$, we can
choose a finite partition
$$
0 \, = \, t_0 \, < \, t_1 \, < \, \dots \, < \, t_n \, = \, t
$$
such that $2(t_{j+1} - t_j) < 1 - t_j$ for every $j$. After $n$ steps, we obtain that $h_t \in [J']$.

Let us examine the statement (\ref{eq:required-convergence}) more closely. Expanding definitions, we need to show that
$$
J'(z) q(z)^t q(rz)^s \to J'(z) q(z)^{s+t} \qquad \text{in }A^2_1.
$$
The convergence is uniform on compact subsets of the unit disk.
To upgrade this to convergence in the $A^2_1$ norm, we split the unit disk into the ball $B(0,r)$ and the annulus $A(0; r, 1)$. 

\subsection{Estimate on the ball $B(0,r)$}

We now show that
\begin{equation}
\label{eq:step1}
\bigl \| \chi_{B(0,r)} (h_t a_{r,s} - h_{t+s}) \bigr \|_{A^2_1} \to 0,
\end{equation}
as $r \to 1$.

For $z \in B(0,r)$, the hyperbolic distance between $z$ and $rz$ is uniformly bounded. Hence by Lemma
\ref{flexibility},
$$
1 - |J(rz)|^2 \asymp 1 - |J(z)|^2.
$$
Using the estimate (\ref{eq:qbound}), we obtain
$$
| J'(z) q(z)^t q(rz)^s |^2 (1-|z|^2) \lesssim |J'(z)|^2 (1-|J(z)|^2)^{-2(t+s)} (1-|z|^2)
$$
and
$$
| J'(z) q(z)^{s+t} |^2 (1-|z|^2) \lesssim |J'(z)|^2 (1-|J(z)|^2)^{-2(t+s)} (1-|z|^2),
$$
for $z \in B(0, r)$. The elementary estimate $|u - v|^2 \le 2|u|^2 + 2|v|^2$ shows
$$
\bigl | J'(z) q(z)^t q(rz)^s - J'(z) q(z)^{s+t} \bigr |^2 (1-|z|^2) \, \lesssim \, |J'(z)|^2 (1-|J(z)|^2)^{-2(t+s)} (1-|z|^2),
$$
for $z \in B(0,r)$. Notice that the right hand side is independent of $0 < r < 1$. By Lemma \ref{beta-lemma} and the bound on the step size  (\ref{eq:st}), the right hand side is integrable. 
From here, (\ref{eq:step1}) follows from the dominated convergence theorem.

\subsection{Estimate on the annulus $A(0; r,1)$} 

We cut the annulus $A(0, r, 1)$ into Carleson squares $Q_j$ of side length $h = 1-r$. 
Set
$$
D_J=1-|J|^2,\qquad D_F=1-|F|^2,
$$
and
$$
m_j=\int_{Q_j}|h_t(z)|^2(1-|z|^2)\,dA(z).
$$
Applying H\"older's inequality on each Carleson square and using Lemma \ref{nde}, we get
\begin{align}
\notag
m_j
& \le  \biggl (\int_{Q_j} |J'(z)|^2 (1-|z|^2) dA(z)\biggr )^{1-t} \biggl (\int_{Q_j} |F'(z)|^2 (1-|z|^2) dA(z)\biggr )^t \\
\notag
& \lesssim h \, D_J(z_{Q_j})^{1-t} \, D_F(z_{Q_j})^t \\
\notag
& \lesssim h \, D_J(z_{Q_j})^{1-t}.
\end{align}
Thus, if $p = \frac{2s}{1-t},$ then
\begin{equation}
\label{eq:ingredient}
m_j^p  \lesssim h^p \, D_J(z_{Q_j})^{2s} \qquad \text{and hence} \qquad D_J(z_{Q_j})^{-2s} m_j \lesssim h^p \, m_j^{1-p}.
\end{equation}
By Lemma \ref{flexibility} and (\ref{eq:qbound}), we have
$$
|q(rz)|^{2s} \lesssim D_J(z_{Q_j})^{-2s}, \qquad z \in Q_j.
$$

Summing over the Carleson squares that make the annulus $A(0, r, 1)$, we get
\begin{align*}
\int_{A_r} |h_t(z)|^2 |q(rz)|^{2s} (1-|z|^2) dA(z) 
& \lesssim \sum_j D_J(z_{Q_j})^{-2s} m_j  \\
& \lesssim \sum_j h^p \, m_j^{1-p} \\
& \lesssim \biggl (\sum_j h \biggr )^p \biggl (\sum_j  m_j \biggr )^{1-p}.
\end{align*}
Here, we have used H\"older's inequality for sums with exponents $1/p$ and $1/(1-p)$.

Since the lengths of the sides of the Carleson squares add up to the circumference of the circle, the first term is bounded. Consequently,
$$
\int_{A_r} |h_t(z)|^2 |q(rz)|^{2s} (1-|z|^2) dA(z) \lesssim \biggl ( \int_{A_r} |h_t(z)|^2 (1-|z|^2) dA(z) \biggr )^{1-p}.
$$
As $h_t \in A^2_1$, the last integral tends to 0 as $r \to 1$.

Since $h_{t+s} \in A^2_1$, the absolute continuity of the integral gives
$$
\int_{A_r} |h_{t+s}(z)|^2 (1-|z|^2) dA(z) \to 0, \qquad \text{as } r \to 1.
$$
By the elementary estimate $|u - v|^2 \le 2|u|^2 + 2|v|^2$,
$$
\int_{A_r} |h_t a_{r,s} - h_{t+s}|^2 (1-|z|^2) dA(z) \le 2 \int_{A_r} |h_t(z)|^2 |q(rz)|^{2s} (1-|z|^2) dA(z)
$$
$$
+2 \int_{A_r} |h_{t+s}(z)|^2 (1-|z|^2) dA(z).
$$
Since both terms on the right tend to 0,
$$
\bigl \| \chi_{A(0;r,1)} (h_ta_{r,s}-h_{t+s} ) \bigr \|_{A^2_1}
\to 0,
$$
as $r \to 1$. Together with (\ref{eq:step1}), this shows (\ref{eq:required-convergence}) as desired.

\subsection{Application to invariant subspaces}

We now deduce Theorem \ref{main-thm} from Theorem \ref{main-thm2}:

\begin{proof}[Proof of Theorem \ref{main-thm}]
Let $I_0$ and $I$ be Liouville maps for $u_{H,0}$ and $u_{H, \infty}$ respectively. From \cite[Theorem 3.1]{critical-structures}, we know that $I' \in [H]$ while $I'_0$ generates $[H]$. As $u_{H, 0} \le u_{H, \infty}$,
$$
\frac{|I'_0|}{1-|I_0|^2} \le \frac{|I'|}{1-|I|^2}.
$$
Since $I_0$ and $I$ have the same critical set as $H$, the quotient $q = I'_0/I'$ is holomorphic and non-vanishing.
Rearranging, we get
$$
|q| \, \le \, \frac{1-|I_0|^2}{1-|I|^2} \, \le \, \frac{1}{1-|I|^2}.
$$
By Theorem \ref{main-thm2}, $I'_0 \in [I']$. Hence, $[H]= [I'_0] = [I']$.

Finally, we address uniqueness. Suppose $F$ and $G$ are inner functions
with $[F'] = [G'] = [H]$. By the third bullet in Section \ref{sec:canonical-solutions},
$I_{F'} = I_{G'}$, while by the fifth, $I_{F'} = F$ and $I_{G'} = G$.
Therefore, $F$ and $G$ agree up to post-composition with an automorphism
of the unit disk.
\end{proof}

\subsection*{Acknowledgements} This research was supported by the Israel Science Foundation
(grant 3134/21).

\subsection*{Declaration on the use of AI Tools}
The proof of Theorem \ref{main-thm} was found by ChatGPT-6 Astra. The role of the human author is purely expository.

\bibliographystyle{amsplain}

\end{document}